\documentclass[10pt]{article}
\usepackage{amsmath,amscd,amsthm,amssymb}

\newtheorem{theorem}{Theorem}{}
{}
{}
{}
{}

\chardef\No=24

\begin{document}

\begin{center}
{\Large Representations for parameter derivatives of some
Koornwinder polynomials in two variables}
\end{center}

\

{\centerline { Rabia Aktas}}

\

{\centerline {E-mail address: raktas@science.ankara.edu.tr}}

\

\centerline{\it Department of Mathematics, Ankara University, Ankara, Turkey}

\

\begin{abstract}  In this paper, we give the parameter derivative
representations in the form of
\begin{equation*}
\frac{\partial P_{n,k}(\lambda;x,y)}{\partial\lambda}=\sum_{m=0}^{n-1}
\sum_{j=0}^{m}d_{n,j,m}P_{m,j}(\lambda;x,y)
+\sum_{j=0}^{k}e_{n,j,k}P_{n,j}(\lambda;x,y)
\end{equation*}
for some Koornwinder polynomials where $\lambda$ is a parameter and $0\leq
k\leq n$; $n=0,1,2,...$ and present orthogonality properties of the
parametric derivatives of these polynomials.

\


\end{abstract}

\section{Introduction}

Recently, many authors \cite%
{A,A1,Froehlich,Koepf,Koepf-schmersau,Lewanowicz,Ronveaux,Szmytkowski,Szmytkowski5,Szmytkowski6,Szmytkowski4,Szmytkowski1,Szmytkowski2,Szmytkowski3}
have studied the representations for the parameter derivatives of the
classical orthogonal polynomials and various special functions which have
many applications in applied mathematics, mathematical and theoretical
physics and many branches of mathematics. In \cite{Szmytkowski5,Szmytkowski6}%
, the derivative of the Legendre function of the first kind, with respect to
its degree $\nu$, $[{\partial P_{\nu}(z)}/{\partial\nu}]_{\nu=n} (n \in {%
\mathbb{N}})$, and its some representations have been examined by
Szmytkowski, which are seen in some engineering and physical problems such
as in the general theory of relativity and in solving some boundary value
problems of potential theory, of electromagnetism and of heat conduction in
solids. In \cite{Szmytkowski4}, explicit expressions of second-order
derivative $[{\partial^{2} P_{\nu}(z)}/{\partial\nu}^{2}]_{\nu=0}$ and of
third-order derivative $[{\partial^{3} P_{\nu}(z)}/{\partial\nu}^{3}]_{\nu=0}
$ have been derived. In \cite{Szmytkowski1,Szmytkowski2,Szmytkowski3}, the
author has presented the derivatives of the associated Legendre function of
the first kind with respect to its order and its degree and also a
relationship between these derivatives. Such derivatives of the associated
Legendre function are met in solutions of various problems of theoretical
acoustics, heat conduction and other branches of theoretical physics. In
\cite{Froehlich,Koepf,Koepf-schmersau,Lewanowicz,Ronveaux,Szmytkowski,Wulkow}%
, the representations of parametric derivatives in the form
\begin{equation}
\frac{\partial P_{n}(\lambda;x)}{\partial\lambda}=\sum_{k=0}^{n}
c_{n,k}(\lambda)P_{k}(\lambda;x)  \label{1.1}
\end{equation}
for orthogonal polynomials in one variable, $\lambda$ being a parameter,
have been studied. For instance, the representations of parametric
derivatives have been obtained by Wulkow \cite{Wulkow} for discrete Laguerre
polynomials, by Froehlich \cite{Froehlich} for Jacobi polynomials $%
P_{n}^{(\alpha,\beta)}(x)$, by Koepf \cite{Koepf} for generalized Laguerre
polynomials $L_{n}^{(\alpha)}(x)$ and Gegenbauer polynomials $%
C_{n}^{(\lambda)}(x)$, by Koepf and Schmersau \cite{Koepf-schmersau} for all
the continuous and discrete classical orthogonal polynomials. In \cite%
{Szmytkowski}, Szmytkowski has derived again the expansions in the form of (%
\ref{1.1}) for Jacobi polynomials, Gegenbauer polynomials and the
generalized Laguerre polynomials by means of a method which is different
from the methods given by Froehlich \cite{Froehlich} and Koepf \cite{Koepf}.
In \cite{Ronveaux}, Ronveaux \emph{et al.} have presented the recurrence
relations for coefficients in the expansion
\begin{equation*}
\frac{\partial^{m} P_{n}(\lambda;x)}{\partial\lambda^{m}}=\sum_{k=0}^{n}
a_{n,k}(m,\lambda)P_{k}(\lambda;x)\qquad (m \in {\mathbb{N}})
\end{equation*}
which is more general than the expansion form of (\ref{1.1}). Moreover,
Lewanowicz \cite{Lewanowicz} has given a method to obtain iteratively
explicit parameter derivative representations of order $m=1,2,...$ for
almost all the classical orthogonal polynomial families, i.e., continuous
classical orthogonal polynomials, classical orthogonal polynomials of a
discrete variable or q-classical orthogonal polynomials of the Hahn's class.
\newline
The classical Jacobi polynomials $P_{n}^{(\alpha,\beta)}(x)$ are defined by
the Rodrigues formula
\begin{equation*}
P_{n}^{(\alpha,\beta)}(x)=\frac{(-1)^{n}}{2^{n}n!}(1-x)^{-{\alpha}}(1+x)^{-{%
\beta}} \frac{d^{n}}{dx^{n}}\{(1-x)^{n+\alpha}(1+x)^{n+\beta}\}
\end{equation*}
and they satisfy the following orthogonality relation
\begin{eqnarray}
\begin{split}
\int_{-1}^{1}P^{(\alpha,\beta)}_{n}(x)P^{(\alpha,\beta)}_{m}(x)(1-x)^{%
\alpha}(1+x)^{\beta}dx&=\frac{2^{\alpha+\beta+1}\Gamma(\alpha+n+1)\Gamma(%
\beta+n+1)} {n!(\alpha+\beta+2n+1)\Gamma(\alpha+\beta+n+1)}\delta_{n,m} \\
&=d_{n}^{(\alpha,\beta)}\delta_{n,m}  \label{j1}
\end{split}%
\end{eqnarray}
where $\delta_{n,m}$ denotes Kronecker's delta \cite{Rainville}.\newline
The generalized Laguerre polynomials defined by \cite{Rainville}
\begin{equation*}
L_{n}^{(\alpha)}(x)=\frac{x^{-\alpha}e^{x}}{n!} \frac{d^{n}}{dx^{n}}\{e^{-x}
x^{n+\alpha}\}
\end{equation*}
hold
\begin{eqnarray*}
\begin{split}
\int_{0}^{\infty}L^{(\alpha)}_{n}(x)L^{(\alpha)}_{m}(x)e^{-x} x^{\alpha}dx=%
\frac{\Gamma(\alpha+n+1)} {n!}\delta_{n,m}.
\end{split}%
\end{eqnarray*}
The representations of parametric derivatives obtained for the Jacobi
polynomials $P_{n}^{(\alpha,\beta)}$ (\cite{Froehlich}) and generalized
Laguerre polynomials $L_{n}^{(\alpha)}(x)$ (\cite{Koepf}) are as follows
\begin{equation}
\begin{split}
\frac{\partial P_{n}^{(\alpha,\beta)}(x)}{\partial\alpha} &=\sum_{k=0}^{n-1}%
\frac{1}{n+k+\alpha+\beta+1}P_{n}^{(\alpha,\beta)}(x) \\
&+\frac{(\beta+1)_{n}}{(\alpha+\beta+1)_{n}}\sum_{k=0}^{n-1} \frac{%
(2k+\alpha+\beta+1)(\alpha+\beta+1)_{k}} {(n-k)(n+k+\alpha+\beta+1)(%
\beta+1)_{k}}P_{k}^{(\alpha,\beta)}(x)  \label{1.2}
\end{split}%
\end{equation}
and
\begin{equation}
\begin{split}
\frac{\partial P_{n}^{(\alpha,\beta)}(x)}{\partial\beta} &=\sum_{k=0}^{n-1}%
\frac{1}{n+k+\alpha+\beta+1}P_{n}^{(\alpha,\beta)}(x) \\
&+\frac{(\alpha+1)_{n}}{(\alpha+\beta+1)_{n}}\sum_{k=0}^{n-1} \frac{%
(-1)^{n+k}(2k+\alpha+\beta+1)(\alpha+\beta+1)_{k}}{(n-k)(n+k+\alpha+%
\beta+1)(\alpha+1)_{k}} P_{k}^{(\alpha,\beta)}(x)  \label{1.3}
\end{split}%
\end{equation}
for $\alpha,\beta>-1$ and
\begin{equation}
\frac{\partial L_{n}^{(\alpha)}(x)}{\partial \alpha} =\sum_{k=0}^{n-1}\frac{1%
}{n-k}L_{k}^{(\alpha)}(x)  \label{1.3*}
\end{equation}
for $\alpha>-1$ where the Pochhammer symbol is defined by
\begin{equation*}
(\alpha)_{0}=1, \quad (\alpha)_{k}=\alpha(\alpha+1)...(\alpha+k-1),
k=1,2,....
\end{equation*}
With motivation from the expansion (\ref{1.1}) for orthogonal polynomials in
one variable, we consider similar expansion in the form of
\begin{equation}
\frac{\partial P_{n,k}(\lambda;x,y)}{\partial\lambda}=\sum_{m=0}^{n-1}
\sum_{j=0}^{m}d_{n,j,m}P_{m,j}(\lambda;x,y)
+\sum_{j=0}^{k}e_{n,j,k}P_{n,j}(\lambda;x,y)  \label{1.1*}
\end{equation}
for orthogonal polynomials of variables $x$ and $y$, with $\lambda$ being a
parameter and $0\leq k\leq n;\quad n=0,1,2,...$. In the recent papers \cite%
{A,A1}, parametric derivative representations in the form of (\ref{1.1*})
for Jacobi polynomials on the triangle and a family of orthogonal
polynomials with two variables on the unit disc have been studied. The
present paper is devoted to obtain parametric derivatives for the
polynomials on the parabolic biangle, on the square and some new
examples of Koornwinder polynomials introduced in \cite{FPP} (see also \cite%
{Marcellan}). Although the parameter derivatives of these polynomials with
respect to their some parameters are in the form of (\ref{1.1*}), there
exist some other parameters such that derivatives with respect to them are
not in the form of (\ref{1.1*}). Because, some of the coefficients $d_{n,j,m}
$ and $e_{n,j,k}$ depend on the variable $x$. \newline
The set up of this paper is summarized as follows. In section 2, we remind
the method given by Koornwinder \cite{Koornwinder} and some examples of this
method. Section 3 contains parametric derivatives of Koornwinder polynomials
on the parabolic biangle, orthogonal polynomials on the square, Laguerre-
Jacobi Koornwinder polynomials and Laguerre-Laguerre Koornwinder
polynomials. In section 4, some orthogonality relations for the derivatives
of these polynomials are studied.

\section{Preliminaries}

First we recall some basic properties of orthogonal polynomials in two
variables \cite{DX}.\newline
Let $\Pi$ be the set of all polynomials in two variables and let $\Pi_{n}$
denote the linear space of polynomials in two variables of total degree at
most $n$.\newline
A polynomial $p\in \Pi_{n}$ is called an orthogonal polynomial with respect
to the weight function $\omega(x,y)$ if
\begin{equation*}
{\Big\langle p,q\Big\rangle}:=\int_{\Omega}p(x,y) q(x,y)\omega(x,y) dxdy=0
\end{equation*}
for all $q\in \Pi_{n-1}$. Let $V_{n}$ denote the space of orthogonal
polynomials of degree $n$ with respect to ${\langle ,\rangle}$.\newline
In 1975, T. H. Koornwinder \cite{Koornwinder} constructed the following
method to derive orthogonal polynomials in two variables from orthogonal
polynomials in one variable. \newline
Let $\omega_{1}(x)$ and $\omega_{2}(y)$ be univariate weight functions
defined on the intervals $(a, b)$ and $(c, d)$, respectively. Let $\rho(x)$
be a positive function on $(a, b)$ which is either a polynomial of degree $r
, (r=0,1,...)$ or the square root of a polynomial of degree $2r\quad(r=\frac{%
1}{2},1,\frac{3}{2},...)$. If $\rho(x)$ is not a polynomial , $c=-d<0$ and $%
\omega_{2}(y)$ is an even function on $(-d,d)$. For $k\geq0$, let ${%
p_{n}(x;k), n=0,1,...}$ be orthogonal polynomial respect to the weight
function $\rho^{2k+1}(x) \omega_{1}(x)$ and let ${q_{n}(y)}, n\geq0$ be
orthogonal polynomial with respect to the weight function $\omega_{2}(y)$.
Then, the family of polynomials
\begin{equation*}
P_{n,k}(x,y)=p_{n-k}(x;k)\rho^{k}(x) q_{k}(\frac{y}{\rho(x)}), \quad 0\leq
k\leq n
\end{equation*}
are orthogonal with respect to the Koornwinder weight function $%
\omega(x,y)=\omega_{1}(x)\omega_{2}(\frac{y}{\rho(x)})$ over the domain
\begin{equation*}
\Omega=\{(x,y): a\leq x \leq b, c\rho(x) \leq y \leq d\rho(x)\}
\end{equation*}
with respect to the inner product
\begin{equation*}
{\Big\langle f, g \Big\rangle}:=\int_{\Omega}f(x,y) g(x,y)\omega(x,y)dxdy.
\end{equation*}
\newline
Some examples of Koornwinder's method are as follows:\newline
(i) Orthogonal polynomials on the parabolic biangle:\newline
For $\alpha,\beta>-1$, Koornwinder polynomials on the parabolic biangle $%
\Omega=\{(x,y): y^{2}\leq x \leq 1\}$ correspond with
\begin{equation*}
\omega_{1}(x)=(1-x)^{\alpha}x^{\beta}, \quad 0 \leq x \leq 1,
\end{equation*}
\begin{equation*}
\omega_{2}(y)=(1-y^{2})^{\beta}, \quad -1 \leq y \leq 1,
\end{equation*}
\begin{equation*}
\rho(x)=\sqrt{x}.~~~~~~~~~~~~~~~~~~~~~~~~~
\end{equation*}
These polynomials can be defined as
\begin{equation}  \label{t2}
P^{(\alpha,\beta)}_{n,k}(x,y)=P^{(\alpha,\beta+k+\frac{1}{2}%
)}_{n-k}(2x-1)x^{k/2}P^{(\beta,\beta)}_{k}(\frac{y}{\sqrt{x}}), \quad 0\leq
k \leq n
\end{equation}
and they are orthogonal with respect to the weight function
\begin{equation*}
\omega(x,y)=(1-x)^{\alpha}(x-y^{2})^{\beta}.
\end{equation*}
In fact,
\begin{eqnarray}
\begin{split}
{\Big\langle P^{(\alpha,\beta)}_{n,k}(x,y), P^{(\alpha,\beta)}_{m,j}(x,y)%
\Big\rangle}&:=\int_{\Omega}P^{(\alpha,\beta)}_{n,k}(x,y)
P^{(\alpha,\beta)}_{m,j}(x,y)(1-x)^{\alpha}(x-y^{2})^{\beta} dxdy \\
&=h_{n,k}^{(\alpha,\beta)}\delta_{n,m}\delta_{k,j}  \label{K11}
\end{split}%
\end{eqnarray}
where
\begin{equation}  \label{K1}
h_{n,k}^{(\alpha,\beta)}=\frac{2^{2\beta+1}\Gamma^{2}(\beta+k+1)\Gamma(%
\alpha+n-k+1)\Gamma(\beta+n+\frac{3}{2})} {(n-k)!k!(2\beta+2k+1)\Gamma(2%
\beta+k+1)\Gamma(\alpha+\beta+n+\frac{3}{2})(\alpha+\beta+2n-k+\frac{3}{2})}.
\end{equation}
(ii) Orthogonal polynomials on the square:\newline
For $\alpha,\beta,\gamma,\delta>-1$, the polynomials defined by
\begin{equation}  \label{t10}
P^{(\alpha,\beta,\gamma,\delta)}_{n,k}(x,y)=P^{(\alpha,\beta)}_{n-k}(x)P^{(%
\gamma,\delta)}_{k}(y), \quad 0\leq k \leq n
\end{equation}
are orthogonal with respect to the weight function $\omega(x,y)=(1-x)^{%
\alpha}(1+x)^{\beta}(1-y)^{\gamma}(1+y)^{\delta}$ on the square $%
\Omega=\{(x,y): -1\leq x \leq 1, -1\leq y \leq 1\}.$ In fact,
\begin{eqnarray}
\begin{split}
{\Big\langle P^{(\alpha,\beta,\gamma,\delta)}_{n,k}(x,y),
P^{(\alpha,\beta,\gamma,\delta)}_{m,j}(x,y)\Big\rangle}&:=\int_{\Omega}P^{(%
\alpha,\beta,\gamma,\delta)}_{n,k}(x,y)
P^{(\alpha,\beta,\gamma,\delta)}_{m,j}(x,y)(1-x)^{\alpha}(1+x)^{%
\beta}(1-y)^{\gamma}(1+y)^{\delta} dxdy \\
&=d_{n-k}^{(\alpha,\beta)}d_{k}^{(\gamma,\delta)}\delta_{n,m}\delta_{k,j}
\label{t10*}
\end{split}%
\end{eqnarray}
where $d_{n}^{(\alpha,\beta)}$ is given by (\ref{j1}). \newline
\newline
Some new examples of Koornwinder polynomials were introduced in \cite{FPP}
by using Koornwinder construction. These cases are as follows:\newline
(iii) Laguerre-Jacobi Koornwinder polynomials:\newline
The case of
\begin{equation*}
\omega_{1}(x)=x^{\alpha}e^{-x}, \quad 0 \leq x <\infty,
\end{equation*}
\begin{equation*}
\omega_{2}(y)=(1-y)^{\beta}, \quad -1\leq y \leq 1,
\end{equation*}
\begin{equation*}
\rho(x)=x ~~~~~~~~~~~~~~~~~~~~~~~~~
\end{equation*}
leads to the polynomials
\begin{equation}  \label{t4}
P^{(\alpha,\beta)}_{n,k}(x,y)=L^{(\alpha+2k+1)}_{n-k}(x)x^{k}P^{(%
\beta,0)}_{k}(\frac{y}{x}),\quad 0\leq k \leq n
\end{equation}
which are orthogonal with respect to the weight function $%
\omega(x,y)=x^{\alpha-\beta}e^{-x}(x-y)^{\beta}, (\alpha,\beta>-1)$ over the
domain $\Omega=\{(x,y): -x<y<x, x>0\}$. The following relation holds
\begin{eqnarray*}
\begin{split}
{\Big\langle P^{(\alpha,\beta)}_{n,k}(x,y), P^{(\alpha,\beta)}_{m,j}(x,y)%
\Big\rangle}:&=\int_{\Omega}P^{(\alpha,\beta)}_{n,k}(x,y)
P^{(\alpha,\beta)}_{m,j}(x,y)x^{\alpha-\beta}e^{-x}(x-y)^{\beta} dxdy \\
&=s_{n,k}^{(\alpha,\beta)}\delta_{n,m}\delta_{k,j}
\end{split}%
\end{eqnarray*}
where
\begin{equation}  \label{K2}
s_{n,k}^{(\alpha,\beta)}=\frac{2^{\beta+1}\Gamma(\alpha+n+k+2)} {%
(n-k)!(\beta+2k+1)}.
\end{equation}
\newline
(iv) Laguerre-Laguerre Koornwinder polynomials:\newline
From the Koornwinder construction with
\begin{equation*}
\omega_{1}(x)=x^{\alpha}e^{-x},\quad 0 \leq x <\infty,\quad\alpha>-1
\end{equation*}
\begin{equation*}
\omega_{2}(y)=y^{\beta}e^{-y}, \quad 0\leq y <\infty,\quad\beta>-1
\end{equation*}
\begin{equation*}
\rho(x)=x,\quad \alpha-\beta>-1,~~~~~~~~~~~~~~~~~~~~~~~~~
\end{equation*}
the Laguerre-Laguerre Koornwinder polynomials defined by
\begin{equation}  \label{t5}
P^{(\alpha,\beta)}_{n,k}(x,y)=L^{(\alpha+2k+1)}_{n-k}(x)x^{k}L^{(\beta)}_{k}(%
\frac{y}{x}),\quad 0\leq k \leq n
\end{equation}
are orthogonal with respect to the weight function $\omega(x,y)=x^{\alpha-%
\beta}y^{\beta}e^{-(x+y/x)}$ over the domain
\begin{equation*}
\Omega=\{(x,y): 0\leq x<\infty, 0\leq y<\infty\}.
\end{equation*}
It follows that
\begin{eqnarray*}
\begin{split}
{\Big\langle P^{(\alpha,\beta)}_{n,k}(x,y), P^{(\alpha,\beta)}_{m,j}(x,y)%
\Big\rangle}:=\int_{\Omega}P^{(\alpha,\beta)}_{n,k}(x,y)
P^{(\alpha,\beta)}_{m,j}(x,y)x^{\alpha-\beta}y^{\beta}e^{-(x+y/x)}dxdy
=t_{n,k}^{(\alpha,\beta)}\delta_{n,m}\delta_{k,j}
\end{split}%
\end{eqnarray*}
where
\begin{equation}  \label{K3}
t_{n,k}^{(\alpha,\beta)}=\frac{\Gamma(\beta+k+1)\Gamma(\alpha+n+k+2)} {%
k!(n-k)!}.
\end{equation}
\newline

\section{Parametric derivatives of some Koornwinder polynomials}

In \cite{A,A1}, parametric derivative representations in the form of (\ref%
{1.1*}) for Jacobi polynomials on the triangle and a family of orthogonal
polynomials with two variables on the unit disc have been studied. In this
section, we derive parameter derivatives of Koornwinder polynomials on the
parabolic biangle, on the square and some new examples of Koornwinder
polynomials introduced in \cite{FPP} (see also \cite{Marcellan}). Since
variable $x$ is included in some of the coefficients, there exist some
parameter derivatives such that they are not in the form (\ref{1.1*}). Now,
we consider such representations of parameter derivatives.

\begin{theorem}
\label{Theorem 1.2} For the Koornwinder polynomials over the parabolic
biangle $P^{(\alpha,\beta)}_{n,k}(x,y)$ defined by (\ref{t2}), the parameter
derivative with respect to the parameter $\alpha$ is as follows
\begin{eqnarray}
\begin{split}  \label{a100*}
\frac{\partial}{\partial \alpha} P^{(\alpha,\beta)}_{n,k}(x,y)=&%
\sum_{s=0}^{n-k-1}\frac{1} {\alpha+\beta+n+s+\frac{3}{2}}P^{(\alpha,%
\beta)}_{n,k}(x,y) \\
&+\sum_{s=0}^{n-k-1}\frac{(\alpha+\beta+2n-k-2s-\frac{1}{2})(\beta+n-s+\frac{%
1}{2})_{s+1}}{(s+1)(\alpha+\beta+2n-k-s+\frac{1}{2})(\alpha+\beta+n-s+\frac{1%
}{2})_{s+1}} P^{(\alpha,\beta)}_{n-s-1,k}(x,y)
\end{split}%
\end{eqnarray}
for $n\geq k+1$, $k\geq 0$ and $\frac{\partial}{\partial \alpha}
P^{(\alpha,\beta)}_{n,n}(x,y)=0$ for $n=k\geq 0.$
\end{theorem}

\begin{proof}
If we differentiate the both side of (\ref{t2}) with respect to the parameter $\alpha$, we get
\begin{eqnarray*}
\begin{split}
\frac{\partial}{\partial \alpha}P^{(\alpha,\beta)}_{n,k}(x,y)=x^{k/2}P^{(\beta,\beta)}_{k}(\frac{y}{\sqrt{x}})\frac{\partial}{\partial \alpha}P^{(\alpha,\beta+k+\frac{1}{2})}_{n-k}(2x-1)
\end{split}
\end{eqnarray*}
By using (\ref{1.2}), it concludes that for $n\geq k+1$, $k\geq 0$
\begin{eqnarray*}
\begin{split}
\frac{\partial}{\partial \alpha} P^{(\alpha,\beta)}_{n,k}(x,y)=&\sum_{s=0}^{n-k-1}\frac{1}
{\alpha+\beta+n+s+\frac{3}{2}}P^{(\alpha,\beta)}_{n,k}(x,y)\\
&+\sum_{s=0}^{n-k-1}\frac{(\alpha+\beta+2n-k-2s-\frac{1}{2})(\beta+n-s+\frac{1}{2})_{s+1}}{(s+1)(\alpha+\beta+2n-k-s+\frac{1}{2})(\alpha+\beta+n-s+\frac{1}{2})_{s+1}}
P^{(\alpha,\beta)}_{n-s-1,k}(x,y).
\end{split}
\end{eqnarray*}
It is obvious from (\ref{t2}) that for $n=k\geq 0$, $\frac{\partial}{\partial \alpha} P^{(\alpha,\beta)}_{n,n}(x,y)=0.$
 \end{proof}

\begin{theorem}
\label{Theorem 1.22} For $\alpha,\beta,\gamma,\delta>-1$, the polynomials on
the square defined by (\ref{t10}) satisfy
\begin{equation}
\begin{split}
\frac{\partial P_{n,k}^{(\alpha,\beta,\gamma,\delta)}(x,y)}{\partial\alpha}
&=\sum_{s=0}^{n-k-1}\frac{1}{n-k+s+\alpha+\beta+1}P_{n,k}^{(\alpha,\beta,%
\gamma,\delta)}(x,y) \\
&+\sum_{s=0}^{n-k-1} \frac{(\alpha+\beta+2n-2k-2s-1)(\beta+n-k-s)_{s+1}} {%
(s+1)(\alpha+\beta+2n-2k-s)(\alpha+\beta+n-k-s)_{s+1}}P_{n-s-1,k}^{(\alpha,%
\beta,\gamma,\delta)}(x,y),  \label{1.23}
\end{split}%
\end{equation}

\begin{equation}
\begin{split}
\frac{\partial P_{n,k}^{(\alpha,\beta,\gamma,\delta)}(x,y)}{\partial\beta}
&=\sum_{s=0}^{n-k-1}\frac{1}{n-k+s+\alpha+\beta+1}P_{n,k}^{(\alpha,\beta,%
\gamma,\delta)}(x,y) \\
&+\sum_{s=0}^{n-k-1} \frac{(-1)^{n-k-s}(\alpha+\beta+2n-2k-2s-1)(%
\alpha+n-k-s)_{s+1}} {(s+1)(\alpha+\beta+2n-2k-s)(\alpha+\beta+n-k-s)_{s+1}}%
P_{n-s-1,k}^{(\alpha,\beta,\gamma,\delta)}(x,y)  \label{1.24}
\end{split}%
\end{equation}
for $n\geq k+1, k\geq 0$ and
\begin{equation*}
\frac{\partial P_{n,n}^{(\alpha,\beta,\gamma,\delta)}(x,y)}{\partial\alpha}=%
\frac{\partial P_{n,n}^{(\alpha,\beta,\gamma,\delta)}(x,y)}{\partial\beta}=0
\end{equation*}
for $n=k\geq 0.$ Also,
\begin{equation}
\begin{split}
\frac{\partial P_{n,k}^{(\alpha,\beta,\gamma,\delta)}(x,y)}{\partial\gamma}
&=\sum_{s=0}^{k-1}\frac{1}{\gamma+\delta+k+s+1}P_{n,k}^{(\alpha,\beta,%
\gamma,\delta)}(x,y) \\
&+\sum_{s=0}^{k-1} \frac{(\gamma+\delta+2k-2s-1)(\delta+k-s)_{s+1}} {%
(s+1)(\gamma+\delta+2k-s)(\gamma+\delta+k-s)_{s+1}}P_{n-s-1,k-s-1}^{(\alpha,%
\beta,\gamma,\delta)}(x,y),  \label{1.25}
\end{split}%
\end{equation}
and
\begin{equation}
\begin{split}
\frac{\partial P_{n,k}^{(\alpha,\beta,\gamma,\delta)}(x,y)}{\partial\delta}
&=\sum_{s=0}^{k-1}\frac{1}{\gamma+\delta+k+s+1}P_{n,k}^{(\alpha,\beta,%
\gamma,\delta)}(x,y) \\
&+\sum_{s=0}^{k-1} \frac{(-1)^{k-s}(\gamma+\delta+2k-2s-1)(\gamma+k-s)_{s+1}%
} {(s+1)(\gamma+\delta+2k-s)(\gamma+\delta+k-s)_{s+1}}P_{n-s-1,k-s-1}^{(%
\alpha,\beta,\gamma,\delta)}(x,y),  \label{1.26}
\end{split}%
\end{equation}
for $n\geq k\geq 1$ and
\begin{equation*}
\frac{\partial P_{n,0}^{(\alpha,\beta,\gamma,\delta)}(x,y)}{\partial\gamma}=%
\frac{\partial P_{n,0}^{(\alpha,\beta,\gamma,\delta)}(x,y)}{\partial\delta}%
=0
\end{equation*}
for $n\geq 0, k=0.$
\end{theorem}

\begin{proof}
In view of equalities (\ref{1.2}) and (\ref{1.3}), the proof is clear.
\end{proof}
Now, we can get similar results for Laguerre-Jacobi Koornwinder and Laguerre-Laguerre Koornwinder polynomials.
\begin{theorem}
\label{Theorem1.3} The representations of parameter derivatives with respect
to the parameters $\alpha$ and $\beta$ for Laguerre-Jacobi Koornwinder
polynomials $P^{(\alpha,\beta)}_{n,k}(x,y)$ defined by (\ref{t4}) are given
by
\begin{equation}  \label{a101*}
\frac{\partial}{\partial \alpha} P^{(\alpha,\beta)}_{n,k}(x,y)=%
\sum_{s=0}^{n-k-1}\frac{1} {n-k-s}P^{(\alpha,\beta)}_{k+s,k}(x,y)
\end{equation}
for $n\geq k+1$, $k\geq 0$ and $\frac{\partial}{\partial \alpha}
P^{(\alpha,\beta)}_{n,n}(x,y)=0$ for $n=k\geq 0$. Similarly,
\begin{eqnarray*}
\begin{split}
\frac{\partial}{\partial \beta}P^{(\alpha,\beta)}_{n,k}(x,y)=&
\sum_{s=0}^{k-1}\frac{1} {\beta+k+s+1} P^{(\alpha,\beta)}_{n,k}(x,y) \\
&+\sum_{s=0}^{k-1}\frac{(\beta+2k-2s-1)(k-s)_{s+1}} {(s+1)(\beta+2k-s)(%
\beta+k-s)_{s+1}}x^{s+1}P^{(\alpha+2s+2,\beta)}_{n-s-1,k-s-1}(x,y)
\end{split}%
\end{eqnarray*}
for $n\geq k\geq 1$ and $\frac{\partial}{\partial \beta}P^{(\alpha,%
\beta)}_{n,0}(x,y)=0$ for $n\geq 0, k=0$. It is seen that the parametric
derivative with respect to the parameter $\beta$ is not in the form of (\ref%
{1.1*}) since the coefficients include variable $x$.
\end{theorem}

\begin{theorem}
\label{Theorem1.4} For Laguerre-Laguerre Koornwinder polynomials defined by (%
\ref{t5}), we have
\begin{equation}  \label{a10*}
\frac{\partial}{\partial \alpha} P^{(\alpha,\beta)}_{n,k}(x,y)=%
\sum_{s=0}^{n-k-1}\frac{1} {n-k-s}P^{(\alpha,\beta)}_{k+s,k}(x,y)
\end{equation}
for $n\geq k+1$, $k\geq0$ and $\frac{\partial}{\partial \alpha}
P^{(\alpha,\beta)}_{n,n}(x,y)=0$ for $n=k\geq 0$. Similarly, for $n\geq
k\geq 1$
\begin{equation*}
\frac{\partial}{\partial \beta}P^{(\alpha,\beta)}_{n,k}(x,y)=\sum_{s=0}^{k-1}%
\frac{1} {s+1} x^{s+1}P^{(\alpha+2s+2,\beta)}_{n-s-1,k-s-1}(x,y)
\end{equation*}
which is different from the form of (\ref{1.1*}) since the coefficients
include variable $x$. Also, $\frac{\partial}{\partial \beta}%
P^{(\alpha,\beta)}_{n,0}(x,y)=0$ for $n\geq 0, k=0$.
\end{theorem}

\section{Orthogonality properties of parametric derivatives}

Now, we consider orthogonality properties for the parametric derivatives of
the polynomials on the parabolic biangle, the polynomials on
the square, Laguerre-Jacobi and Laguerre-Laguerre Koornwinder polynomials.

\begin{theorem}
\label{Theorem3} For Koornwinder polynomials on the parabolic biangle given
by (\ref{t2}) and their derivative with respect to the parameter $\alpha$,
we have for $n\geq k+1, k\geq 0$, $0\leq j\leq m; n,m\in \mathbb{N}_{0}$
\begin{equation*}
{\Big\langle P_{m,j}^{(\alpha,\beta)},\frac{\partial}{\partial\alpha}%
P_{n,k}^{(\alpha,\beta)}\Big\rangle}=\left\{
\begin{array}{cc}
0 , & \text{$k\neq j$} \\
0 , & \text{$k=j, m>n$} \\
&  \\
A_{n,k,m}^{(\alpha,\beta)} , & \text{$n>m\geq k=j$} \\
&  \\
B_{n,k}^{(\alpha,\beta)} , & \text{$k=j, n=m$}%
\end{array}%
\right.
\end{equation*}
and for $n=k\geq 0$
\begin{equation*}
{\Big\langle P_{m,j}^{(\alpha,\beta)},\frac{\partial}{\partial\alpha}
P_{n,n}^{(\alpha,\beta)}\Big\rangle}=0
\end{equation*}
where
\begin{eqnarray*}
A_{n,k,m}^{(\alpha,\beta)}= \frac{(\alpha+\beta-k+2m+%
\frac{3}{2})(\beta+m+\frac{3}{2})_{n-m}}{(n-m)(\alpha+\beta+n-k+m+\frac{3%
}{2})(\alpha+\beta+m+\frac{3}{2})_{n-m}} h_{m,k}^{(\alpha,\beta)}
\end{eqnarray*}
and
\begin{eqnarray*}
B_{n,k}^{(\alpha,\beta)}=\sum_{s=0}^{n-k-1}\frac{1}{\alpha+\beta+n+s+\frac{3%
}{2}} h_{n,k}^{(\alpha,\beta)}
\end{eqnarray*}
where $h_{n,k}^{(\alpha,\beta)}$ is given by (\ref{K1}).
\end{theorem}

\begin{proof}
We will divide the proof into two cases.\\
\textbf{Case 1.} We consider the case $n=k\geq 0$. It is seen that
$$
{\Big\langle P_{m,j}^{(\alpha,\beta)},\frac{\partial}{\partial\alpha} P_{n,n}^{(\alpha,\beta)}\Big\rangle}=0
$$
since $\frac{\partial}{\partial\alpha}P_{n,n}^{(\alpha,\alpha)}(x,y)=0$ for $n=k\geq 0$.\\
\textbf{Case 2.} We assume that $n\geq k+1, k\geq 0$. From (\ref{a100*}), we can write
\begin{eqnarray}
\begin{split}
{\Big\langle P_{m,j}^{(\alpha,\beta)},\frac{\partial}{\partial\alpha} P_{n,k}^{(\alpha,\beta)}\Big\rangle}
&=\sum_{s=0}^{n-k-1}\frac{1}{\alpha+\beta+n+s+\frac{3}{2}}{\Big\langle P_{m,j}^{(\alpha,\beta)},P_{n,k}^{(\alpha,\beta)}\Big\rangle}\\
&+\sum_{s=0}^{n-k-1} \frac{(\alpha+\beta+2n-k-2s-\frac{1}{2})(\beta+n-s+\frac{1}{2})_{s+1}}{(s+1)(\alpha+\beta+2n-k-s+\frac{1}{2})(\alpha+\beta+n-s+\frac{1}{2})_{s+1}}{\Big\langle P_{m,j}^{(\alpha,\beta)}, P_{n-s-1,k}^{(\alpha,\beta)}\Big\rangle}.\label{2.17}
\end{split}
\end{eqnarray}
For this case, we consider three subcases.\\
\textbf{Case 2.1.} Let consider the case $k\neq j$ or $k=j, m>n$. It follows from (\ref{K11}) that
\begin{eqnarray*}
{\Big\langle P_{m,j}^{(\alpha,\beta)},\frac{\partial}{\partial\alpha} P_{n,k}^{(\alpha,\beta)}\Big\rangle}=0.
\end{eqnarray*}
\\
\textbf{Case 2.2.} Assume that $k=j, n>m$. Since the first inner product in the right-hand side of the equality (\ref{2.17}) from (\ref{K11}) is zero, we get
\begin{eqnarray*}
{\Big\langle P_{m,k}^{(\alpha,\beta)},\frac{\partial}{\partial\alpha} P_{n,k}^{(\alpha,\beta)}\Big\rangle}
=\sum_{s=0}^{n-k-1} \frac{(\alpha+\beta+2n-k-2s-\frac{1}{2})(\beta+n-s+\frac{1}{2})_{s+1}}{(s+1)(\alpha+\beta+2n-k-s+\frac{1}{2})(\alpha+\beta+n-s+\frac{1}{2})_{s+1}}
h_{n-s-1,k}^{(\alpha,\beta)}\delta_{n-s-1,m},
\end{eqnarray*}
which contains only one non-vanishing term with $s=n-m-1$ for $m\geq k$. One may deduce that
\begin{eqnarray*}
A_{n,k,m}^{(\alpha,\beta)}= \frac{(\alpha+\beta-k+2m+%
\frac{3}{2})(\beta+m+\frac{3}{2})_{n-m}}{(n-m)(\alpha+\beta+n-k+m+\frac{3%
}{2})(\alpha+\beta+m+\frac{3}{2})_{n-m}} h_{m,k}^{(\alpha,\beta)}
\end{eqnarray*}
where $h_{n,k}^{(\alpha,\beta)}$ is given by (\ref{K1}).\\
\textbf{Case 2.3} Let $k=j, m=n$. Then, from the relation (\ref{K11}) we have
\begin{eqnarray*}
{\Big\langle P_{n,k}^{(\alpha,\beta)},\frac{\partial}{\partial\alpha} P_{n,k}^{(\alpha,\beta)}\Big\rangle}
=\sum_{s=0}^{n-k-1}\frac{1}{\alpha+\beta+n+s+\frac{3}{2}}{\Big\langle P_{n,k}^{(\alpha,\beta)},P_{n,k}^{(\alpha,\beta)}\Big\rangle}=\sum_{s=0}^{n-k-1}\frac{1}{\alpha+\beta+n+s+\frac{3}{2}}
h_{n,k}^{(\alpha,\beta)},
\end{eqnarray*}
 which completes the proof.
\end{proof}
By using the parameter derivatives given in Theorem \ref{Theorem 1.22} and
the relation (\ref{t10*}), the next theorem is readily verified.

\begin{theorem}
\label{Theorem4} For Koornwinder polynomials on the square defined by (\ref%
{t10}), we have for $0\leq j\leq m; n,m\in \mathbb{N}_{0}$, $n\geq k+1,
k\geq 0$
\begin{equation*}
{\Big\langle P_{m,j}^{(\alpha,\beta,\gamma,\delta)},\frac{\partial}{%
\partial\alpha}P_{n,k}^{(\alpha,\beta,\gamma,\delta)}\Big\rangle}=\left\{
\begin{array}{cc}
0 , & \text{$k\neq j$} \\
0 , & \text{$k=j, m>n$} \\
&  \\
C_{n,k,m}^{(\alpha,\beta,\gamma,\delta)} , & \text{$n>m\geq k=j$} \\
&  \\
D_{n,k}^{(\alpha,\beta,\gamma,\delta)} , & \text{$k=j, n=m$}%
\end{array}%
\right.
\end{equation*}
\newline
\begin{equation*}
{\Big\langle P_{m,j}^{(\alpha,\beta,\gamma,\delta)},\frac{\partial}{%
\partial\beta}P_{n,k}^{(\alpha,\beta,\gamma,\delta)}\Big\rangle}=\left\{
\begin{array}{cc}
0 , & \text{$k\neq j$} \\
0 , & \text{$k=j, m>n$} \\
&  \\
E_{n,k,m}^{(\alpha,\beta,\gamma,\delta)} , & \text{$n>m\geq k=j$} \\
&  \\
D_{n,k}^{(\alpha,\beta,\gamma,\delta)} , & \text{$k=j, n=m$}%
\end{array}%
\right.
\end{equation*}
and for $n=k\geq 0$
\begin{equation*}
{\Big\langle P_{m,j}^{(\alpha,\beta,\gamma,\delta)},\frac{\partial}{%
\partial\alpha}P_{n,n}^{(\alpha,\beta,\gamma,\delta)}\Big\rangle}={%
\Big\langle P_{m,j}^{(\alpha,\beta,\gamma,\delta)},\frac{\partial}{%
\partial\beta}P_{n,n}^{(\alpha,\beta,\gamma,\delta)}\Big\rangle}=0
\end{equation*}
where
\begin{eqnarray*}
C_{n,k,m}^{(\alpha,\beta,\gamma,\delta)}=
\frac{%
(\alpha+\beta+2m-2k+1)(\beta-k+m+1)_{n-m}} {(n-m)(\alpha+\beta+n+m-2k+1)(%
\alpha+\beta-k+m+1)_{n-m}} d_{m-k}^{(\alpha,\beta)}d_{k}^{(\gamma,%
\delta)}
\end{eqnarray*}
\begin{eqnarray*}
D_{n,k}^{(\alpha,\beta,\gamma,\delta)}=\sum_{s=0}^{n-k-1}\frac{1}{%
n-k+s+\alpha+\beta+1}d_{n-k}^{(\alpha,\beta)}d_{k}^{(\gamma,\delta)},
\end{eqnarray*}
\begin{eqnarray*}
E_{n,k,m}^{(\alpha,\beta,\gamma,\delta)}=
\frac{%
(-1)^{m-k+1}(\alpha+\beta+2m-2k+1)(\alpha-k+m+1)_{n-m}} {(n-m)(\alpha+\beta+n+m-2k+1)(%
\alpha+\beta-k+m+1)_{n-m}} d_{m-k}^{(\alpha,\beta)}d_{k}^{(\gamma,\delta)}
\end{eqnarray*}
where $d_{n}^{(\alpha,\beta)}$ is defined as in (\ref{j1}).
\end{theorem}
Similarly, using the results in Theorem \ref{Theorem1.3} and Theorem \ref{Theorem1.4}, one can easily
obtain the next results.
\begin{theorem}
\label{Theorem5} For Laguerre-Jacobi Koornwinder polynomials $%
P^{(\alpha,\beta)}_{n,k}(x,y)$ defined by (\ref{t4}), $n,m\in \mathbb{N}%
_{0}, 0\leq j\leq m $, we get for $n\geq k+1, k\geq 0$
\begin{equation*}
{\Big\langle P_{m,j}^{(\alpha,\beta)},\frac{\partial}{\partial\alpha}%
P_{n,k}^{(\alpha,\beta)}\Big\rangle}=\left\{
\begin{array}{cc}
0 , & \text{$k\neq j$} \\
0 , & \text{$k=j, m\geq n$} \\
&  \\
G_{n,k,m}^{(\alpha,\beta)} , & \text{$n>m\geq k=j$}%
\end{array}%
\right.
\end{equation*}
and for $n=k\geq 0$
\begin{equation*}
{\Big\langle P_{m,j}^{(\alpha,\beta)},\frac{\partial}{\partial\alpha}
P_{n,n}^{(\alpha,\beta)}\Big\rangle}=0
\end{equation*}
where
\begin{eqnarray*}
G_{n,k,m}^{(\alpha,\beta)}=\frac{1} {n-m}%
s_{m,k}^{(\alpha,\beta)}
\end{eqnarray*}
where $s_{n,k}^{(\alpha,\beta)}$ is given by (\ref{K2}).
\end{theorem}

\begin{theorem}
\label{Theorem6} For Laguerre-Laguerre Koornwinder polynomials $%
P^{(\alpha,\beta)}_{n,k}(x,y)$, $n,m\in \mathbb{N}_{0}, 0\leq j\leq m $, the
following results hold for $n\geq k+1, k\geq 0$
\begin{equation*}
{\Big\langle P_{m,j}^{(\alpha,\beta)},\frac{\partial}{\partial\alpha}%
P_{n,k}^{(\alpha,\beta)}\Big\rangle}=\left\{
\begin{array}{cc}
0 , & \text{$k\neq j$} \\
0 , & \text{$k=j, m\geq n$} \\
&  \\
H_{n,k,m}^{(\alpha,\beta)} , & \text{$n>m\geq k=j$}%
\end{array}%
\right.
\end{equation*}
and for $n=k\geq 0$
\begin{equation*}
{\Big\langle P_{m,j}^{(\alpha,\beta)},\frac{\partial}{\partial\alpha}
P_{n,n}^{(\alpha,\beta)}\Big\rangle}=0
\end{equation*}
where
\begin{eqnarray*}
H_{n,k,m}^{(\alpha,\beta)}=\frac{1} {n-m}%
t_{m,k}^{(\alpha,\beta)}
\end{eqnarray*}
where $t_{n,k}^{(\alpha,\beta)}$ is defined by (\ref{K3}).
\end{theorem}

\end{document}